\documentclass[pdftex,dvipsnames]{article}

\usepackage[numbers]{natbib}
\usepackage[preprint]{neurips_2021}

\usepackage[utf8]{inputenc} 
\usepackage[T1]{fontenc}    
\usepackage{hyperref}       
\usepackage{url}            
\usepackage{booktabs}       
\usepackage{amsfonts}       
\usepackage{nicefrac}       
\usepackage{microtype}      
\usepackage{xcolor}         
\usepackage{comment} 

\title{A Novel Fast Exact Subproblem Solver for Stochastic Quasi-Newton Cubic Regularized Optimization}

\usepackage{natbib}
\usepackage{authblk}
\usepackage{amsmath}  
\usepackage{amsthm}

\DeclareMathOperator*{\argmin}{\arg\min}   
\usepackage[english]{babel}
\usepackage{mathtools}
\usepackage{algorithm}
\newcommand\defeq{\mathrel{\overset{\makebox[0pt]{\mbox{\normalfont\tiny\sffamily def}}}{=}}}

 \usepackage{algpseudocode}
 \usepackage{algorithmicx}
 \algdef{SE}[DOWHILE]{Do}{doWhile}{\algorithmicdo}[1]{\algorithmicwhile\ #1}%
\usepackage{float}
\usepackage[font=small,labelfont=bf,tableposition=top]{caption}

\usepackage{placeins}
\usepackage[colorinlistoftodos,prependcaption,textsize=tiny]{todonotes}
\usepackage{xargs}                      
\newcommandx{\unsure}[2][1=]{\todo[linecolor=red,backgroundcolor=red!25,bordercolor=red,#1]{#2}}
\newcommandx{\change}[2][1=]{\todo[linecolor=blue,backgroundcolor=blue!25,bordercolor=blue,#1]{#2}}
\newcommandx{\info}[2][1=]{\todo[linecolor=OliveGreen,backgroundcolor=OliveGreen!25,bordercolor=OliveGreen,#1]{#2}}
\newcommandx{\improvement}[2][1=]{\todo[linecolor=Plum,backgroundcolor=Plum!25,bordercolor=Plum,#1]{#2}}
\newcommandx{\josh}[2][1=]{\todo[linecolor=red,backgroundcolor=white!25,bordercolor=red,#1]{#2}}
\newcommandx{\wenwen}[2][1=]{\todo[linecolor=red,backgroundcolor=white!25,bordercolor=green,#1]{#2}}

\author[1, 2, 3]{\textbf{Jarad Forristal}}
\author[2]{\textbf{Joshua Griffin}}
\author[2]{\textbf{Wenwen Zhou}}
\author[2]{\textbf{Seyedalireza Yektamaram}}
\affil[1]{Department of Computer Science, The University of Texas at Austin}
\affil[2]{SAS Institute}
\affil[3]{\texttt{Jarad.Forristal@sas.com}}

\newtheorem{theorem}{Theorem}

\newtheorem{assumption}{Assumption}

\begin{document}

\maketitle

\begin{abstract}
  In this work we describe an Adaptive Regularization using Cubics (ARC) method 
  for large-scale nonconvex unconstrained optimization using Limited-memory Quasi-Newton (LQN)
  matrices. ARC methods are a relatively new family of optimization strategies
  that utilize a cubic-regularization (CR) term in place of 
  trust-regions and line-searches. LQN methods offer a large-scale
  alternative to using explicit second-order information by taking identical inputs to those used by popular first-order methods such as stochastic gradient descent (SGD). Solving the CR subproblem exactly requires Newton's method, yet using properties of the internal structure of LQN matrices, we are able to find exact solutions to the CR subproblem in a matrix-free manner, providing large speedups and scaling into modern size requirements. Additionally, we expand upon previous ARC work and explicitly incorporate first-order updates into our algorithm. We provide experimental results when the SR1 update is used, which show substantial speed-ups and competitive performance compared to Adam and other second order optimizers on deep neural networks (DNNs). We find that our new approach, ARCLQN, compares to modern optimizers with minimal tuning, a common pain-point for second order methods.
\end{abstract}

\section{Introduction}

Scalable second-order methods for training deep learning problems have shown great potential, yet 
ones that build on Hessian-vector products may be more expensive to use and arguably require more sophisticated automatic differentiation architecture be available
than required by SGD and its main variants~\cite{SGDoverview}. 
In this paper, we focus on algorithms that require similar information to popular methods such as SGD,
namely, stochastic gradients calculated on mini-batches of data. Quasi-Newton (QN) methods are a natural higher-level alternative to first-order methods, in that they seek to learn curvature information dynamically from past steps based on gradient information. Thus they can work out of the box in the same settings as SGD with little model-specific coding required.

Similar to Hessian-based approaches, the more popular alternative Kronecker-factored Approximate
Curvature (K-FAC)~\cite{martens2015optimizingKFAC} requires careful encoding of the inverse Fisher matrix block diagonal approximation that is layer dependent. Thus KFAC approaches, unlike QN methods, are tightly coupled with the model type. Limited-memory variants of QN are critical in that they scale to similar dimensions to where SGD is used. Potentially exploiting the curvature information, they may be able to converge in lower a number of steps compared to first-order alternatives \cite{LSR1}. Additionally, such higher-order methods parallelize similarly to modern first-order methods, as large batch sizes allow lower variance estimation of the Hessian, which may motivate training in a large batch distributed setting \cite{LSR1}. 

Note that stochastic QN (SQN) methods strive (though not always~\cite{wills2021stochastic}) to mimic the learning-rate policies used by 
SGD~\cite{byrd2016stochastic, wang2017stochastic, schraudolph2007stochastic} 
and often skip line-searches that are typically used in non-stochastic 
variants~\cite{gill1972quasi,nocedal2006quasi,shi2006convergence}. 
For this reason, a relatively recent alternative to line-search and trust-region approaches known as cubic regularization~\cite{nesterov2006cubic, ARCP1,tripuraneni2018stochastic} is a promising alternative.  In detail, we study the minimization problem of
\begin{equation}
    \underset{s \in \mathbb{R}^n}{\textrm{minimize }} m_k(s) \defeq f(x_k) + s^Tg_k + \frac{1}{2} s^TB_ks + \frac{1}{3} \sigma_k \vert\vert s \vert\vert^3,\label{eq:cr_subproblem}
\end{equation}
for a given $x_k$, where $g_k \defeq \nabla f(x_k)$, $B_k$ is a Hessian approximation, $\sigma_k$ an iteratively chosen adaptive regularization parameter, and $f(x)$ the objective function we want to minimize. Equation~\ref{eq:cr_subproblem} is also known as the CR subproblem. Cubic regularization is popular because it can be shown that if
$\nabla^2 f$ is Lipschitz continuous with constant $L,$ then
$f(x_k+s) \le m_k(s)$ whenever $\sigma_k \ge L$ and 
$B_ks=\nabla^2 f(x)s$~\cite{nesterov2006cubic}. Thus if the Hessian approximation $B_k$ behaves
like $\nabla^2 f(x)$ along the search direction $s,$ the model function $m_k(s)$ becomes an upper bound
on the objective $f(x+s).$  In such cases a line-search would not be needed as reduction in
$m_k(s)$ translates directly into reduction in $f(x+s)$, removing the risk that the
computational work performed reducing $m_k(s)$ is wasted.
Thus if we can reliably have $\sigma_k \ge L_k$ and find a Hessian approximation 
$B_k$ such that $B_k s \approx \nabla^2 f(x) s,$ the need for line-searches and learning rate
scheduling could be greatly reduced.  

We propose an efficient exact solver to Equation~\ref{eq:cr_subproblem} which is tractable in large-scale
optimization problems under near-identical conditions for which SGD itself is commonly applied. As Newton's method corresponds to much of the computation overhead when solving
Equation~\ref{eq:cr_subproblem}, a dense approach such as that
described in~\citet{ARCP1} would be prohibitive.  A plausible
choice would be to replace the dense Cholesky factorizations of~\citet{ARCP1} with corresponding
LQN inversions strategies such as that 
described in~\citet{erway2015efficiently}. A challenge in this approach is that an
application of Newton's method can sometimes take a large number 
of inner iterations, which again induces a prohibitive amount of
computational overhead when compared to classic LQN approaches that only require
a single inversion per outer iteration. 
However, by exploiting properties of LQN methods described in~\citet{erway2015efficiently} 
and~\citet{burdakov2017efficiently}, we can instead perform Newton's method in a reduced
subspace such that the cost per Newton iteration is reduced from ($\mathcal{O}(mn) \times$ number of iterations) to ($\mathcal{O}(m) \times$ number of iterations),
where $m$ is commonly chosen to be an integer between 5 and 20. The full-space solution
to Equation~\ref{eq:cr_subproblem} can then be recovered for a cost identical to that of classic LQM methods.

To the best of our knowledge, all previous attempts to use LQN methods in the context of 
cubic-regularization have necessarily had to change the definition of $m_k(s)$ in order to
find an approximate solution. Remarkably, we present a mechanism for minimizing $m_k(s)$
using similar computational efforts to a single matrix inversion of a shifted
LQN matrix (required at a minimum by all such approaches). Further, we show that
by applying Newton's method in the reduced subspace, we can achieve speed improvements 
of more than 100x over a naive (LQN inversion-based) implementation.  
In the numerical results section we further show that this modification permits the
application of LQN matrices with exact cubic regularization as an practical
alternative step-generator for large scale
DNN optimization problems.

\section{Related Work}

Second-order methods for DNNs are steadily growing more common \cite{KBFGS, LSR1, TRLSR1, uSdLBFGS, JustSample, APOLLO, AdaHessian}. There are a number of ways to approach this problem; a popular approach is using the ARC framework (or some variant). We chose to use a cubic overestimation of the true objective function as a means of regularization. \citet{ARCP1} provides a principled framework for exact solutions to Equation~\ref{eq:cr_subproblem}. We expand upon their work using limited-memory techniques and apply it to modern tasks for which traditional dense QN methods are infeasible. While
recent proposals suggesting the use of LQN matrices in the context of ARC
algorithms exist, to date, to the best of our knowledge, all have modified the definition of $m_k(s)$ in order to make
step generation tractable (with respect to solving Equation~\ref{eq:cr_subproblem})~\cite{ADACN, MEMLESSSR1, ARCsLSR1}.


LSR1 updates are studied in the context of ARC methodology
using a ``memory-less" variant in~\citet{MEMLESSSR1}. As we will describe in
Section~\ref{section:ARC}, all QN methods iteratively update $B_k$ matrices with pairs $(s_k,y_k)$
such that $B_k s_k= y_k,$ where $y_k$ denotes the difference of the corresponding gradients
of the objective functions.  In~\citet{MEMLESSSR1}, $y_k$ is 
a difference of the gradients of $m_k(s).$  There is no proof of why this should improve
the desired convergence property that $B_k \to \nabla^2 f(x_k)$; instead they show that their
new search direction is a descent direction which is used to prove first-order convergence of their
approach.  
 
\citet{ADACN} approximately minimize $m_k(s)$ by simply solving shifted systems of form 
\[(B_k + \sigma \|s_{k-1}\| I) s_{k} = -g,
\]
where the norm of the previous step is used as an estimate for the norm of $\|s_k\|$ to define
the optimal shift.  As described in Theorem~\ref{thm:firstorderCR}
in Section~\ref{section:ARC},
the optimal solution necessarily satisfies a condition of the form $(B_k + \sigma \|s_{k}\| I) s_{k} = -g.$
Since the norm of $\|s_k\|$ can vary significantly each iteration, their approximate solution is unlikely to be close to
the optimal solution of $m_k(s).$ They further simplify the sub-problem using only the diagonals of $B_k + \sigma \|s_{k-1}\| I$
when generating $s_k.$ 

\citet{ARCsLSR1} solves a modified version of the problem using a shape-changing norm as the cubic overestimation that provides an analytical solution to Equation~\ref{eq:cr_subproblem}. They transform
$m_k(s)$ using similar strategies to those advocated in this paper.  By relaxing the definition
of the norm they are able to analytically solve their new problem.  Unfortunately,
this norm definition is dependent on the matrix $B_k$ and thus makes the definition of
the target Lipschitz constant, $L$, dependent as well.  A nontrivial
distinction in our approaches is they require a QR factorization of matrices of size $n\times m.$
This may be prohibitive for deep learning problems, which may have billions of parameters.  We
take a similar approach to that used by~\citet{burdakov2017efficiently} for trust-region methods
and perform such operations on matrices whose row and column dimensions are both $\mathcal{O}(m).$
Unlike the approach advocated in~\citet{burdakov2017efficiently}, we avoid inversions of potentially
ill-conditioned systems to simultaneously improve the stability of the approach and reducing computation
overhead costs at the same time.   

In~\citet{SANC}, the ARC framework with stochastic gradients is used with a Hessian-based approach
first advocated by~\citet{martens2010deep}. In this case, $\nabla^2 f(x)$ is approximated within
a Krylov-based subspace using Hessian-vector projects with batched estimates of $\nabla^2 f(x).$
They then minimize $m_k(s)$ with this small-dimensional subspace.  We note that a downside of Hessian-based approaches compared to LQN methods is the requirement of the existence of analytic second-order vector multiplies which may or may not be optimized for a given problem.

 An alternative to ARC methods is the use
 of trust-regions and line-searches.  Line-searches are applicable whenever the QN update strategy
 can guarantee $B_k$ is positive definite (such as BFGS), while trust-regions handle
 indefinite matrices but have a similar subproblem complexity to that of ARC
 and are thus harder to solve. Though fundamentally different approaches, they
 have some parallels; thus we can often borrow
 technology from the trust-region subproblem solver space to adapt to the ARC context.
For example, \citet{TRLSR1} outlines mechanisms for efficiently computing 
$(B_k+\lambda I)^{-1} g$ and implicit eigen-decomposition of $B_k + \lambda I$ when
solving the trust-region subproblem 
\begin{eqnarray*}
    \underset{s \in \mathbb{R}^n}{\textrm{minimize }} && q_k(s) = f(x_k) + s^Tg_k + \frac{1}{2} s^TB_ks \\
    \text{ subject to} && \|s\| \le \delta
\end{eqnarray*} 
while~\citet{burdakov2017efficiently} significantly reduces the complexity and memory 
cost of such algebraic operations while solving the same problem.  We thus
adopt select operations developed therein when applicable to adapt~\citet{ARCP1} to the
LQN context.
 
Note that we further extend~\citet{ARCP1} to the stochastic optimization setting. 
Thus we also share relation to past stochastic
 QN approaches.
 \citet{erway2020trust} use the tools described in~\citet{TRLSR1} to create 
 a stochastic trust-region solver using LSR1 updates. \citet{oLBFGS} generalizes BFGS and LBFGS to the online convex optimization setting. \citet{RES} studies BFGS applied to the stochastic convex case and develops a regularization scheme to prevent the BFGS matrix from becoming singular. \citet{SFO} explores domain-specific modifications to SGD and BFGS for sum-of-functions minimization, where the objective function is composed of the sum of multiple differential subfunctions. \citet{SQN} considers not using simple gradient differencing for the BFGS update, but instead more carefully picking to use as $(s_k, y_k)$ pairs; \citet{JustSample} also explores a similar idea. \citet{SDLFBFGS} tries to prevent ill-conditioning of $B_k$ for BFGS updates, similar to \citet{RES}, but explicitly for the nonconvex case. \citet{AdaQN} presents is an optimizer designed specifically for RNNs that builds off \citet{SQN}. 
 
 We note that though our approach is targeting LQN matrices in particular, it is also
 valid for any approach that incorporates diagonal approximations to the Hessian such as
 \citet{APOLLO}. In fact, this observation is current future work.

\paragraph{Our Contributions. } 
\begin{enumerate}
    \item A fast $\mathcal{O}(mn)$ approach for exactly solving the cubic 
    regularization problem for any limited memory
    quasi-Newton approximation that lends itself to
    an efficient eigen-decomposition such as LBFGS and LSR1,
    \item A hybrid first and second-order stochastic Quasi-Newton ARC framework that is competitive
    with current SOTA optimizers, and,
    \item Empirical results of this optimizer applied to
    real-life nonconvex problems.
\end{enumerate}

\section{Algorithm}\label{section:ARC}
In this section we describe the proposed algorithm.  We will
first describe how to exactly and efficiently solve Equation~\ref{eq:cr_subproblem}
when $B_k$ is defined by a limited memory Quasi-Newton matrix.  We 
will demonstrate that the computational complexity of ARCLQN is similar to
that of classical unregularized limited memory Quasi-Newton solvers.
Later in this section we describe how to solve the nonlinear optimization 
problem (Algorithm~\ref{algo:simple_ARC}) using the aforementioned subproblem solver.
We first provide a brief overview of Quasi-Newton matrices.

Popular 
Quasi-Newton updates such as BFGS, DFP, and SR1, are based on iteratively updating
an initial matrix $B_0=\gamma I$ with rank one or two corrections with pairs $(s_k,y_k)$ such
that the property $B_k s_k = y_k$ is maintained each update~\cite{nocedal2006quasi}. For
example, the popular SR1 update formula is given by the recursive update formula:
\begin{equation}\label{eq:SR1update}
B_{k + 1} = B_k + \frac{(y_k - B_k s_k)(y_k - B_k s_k)^T}{s^T_k(y_k - B_ks_k)},
\end{equation}
where $s_k$ and $y_k$ are defined as
$
y = g_k - g_{k-1}
\text{ and } s = x_k - x_{k-1}$.

At each iteration, the update is well defined if $s_k^T (y_k - B_ks_k) \ne 0$. Thus,
\begin{equation}\label{eq:curvature}
\|s_k^T (y_k - B_ks_k)\| > \epsilon \|s_k \| \|y_k - B_ks_k\|
\end{equation}
is checked with a small number $\epsilon$. If condition~\ref{eq:curvature} is not satisfied, $B_k$ is not updated. This helps ensure that $B_k$ remain bounded. The motivation of Quasi-Newton
matrices comes from the Taylor expansion of the gradient function $g(x)$
of a univariate function $f(x)$:
\[
g(x+s) \approx g(x) + \nabla^2 f(x) s
\]
This implies that $\nabla^2 f(x) s \approx g(x+s) - g(x).$  It thus makes sense to require
that any approximate Hessian $B$ we use satisfy the same equation $Bs=y$ where we
define $y=g(x+s)-g(x).$  A number of remarkable properties have been proven for 
popular updates such as SR1 and BFGS, which prove convergence to the
true Hessian~\cite{nocedal2006quasi}.  SR1 is a popular choice in the nonconvex 
optimization case as unlike BFGS, SR1 is able to learn negative-curvature without
becoming singular~\cite{erway2020trust}. It was also shown by \citet{convSR1} that the SR1 update converges to the true Hessian whenever the search directions remain independent.  In contrast, LBFGS requires that exact line-search be used to achieve the same convergence results. A compact representation of the SR1 update is detailed in \citet{COMPACTSR1} and explored more in \citet{LSR1}. For more details, we urge readers to consider those papers. Thus in the numerical results section we 
will be primarily focused on the SR1 update; however, we stress that the exact subproblem
solver proposed in this section will hold for all QN variants described in~\citet{erway2015efficiently}.

Note that $B_k$ if is explicitly formed, it requires $\mathcal{O}(n^2)$ space; as such, for large-scale problems, limited-memory variants 
are popular.  For such cases only the most recent $m \ll n$ pairs are stored
in $n \times m$ matrices $S_k$ and $Y_k,$
\[
S_k \defeq (s_{k-m+1}, \ldots , s_k) \text{ and } Y_k \defeq (y_{k-m+1}, \ldots , y_k).
\]
In the limited memory case $B_k$ is never explicitly formed and operations $B_k$ are
performed using only $\gamma,$ $S$, and $Y$ using $\mathcal{O}(mn)$ operations.  
How this is done specifically for the cubic-regularized case will become clearer later in this section.
Before proceeding, we will next briefly describe the 
approach used by~\citet{ARCP1} for the case where $B_k$ is dense.
Later we describe how to adapt their dense approach to the limited-memory case.
\subsection{Solving the cubic regularized sub-problem}\label{sec:subsolver}
In this section we are focused on efficiently finding a global
solution to the cubic regularized subproblem given in Equation~\ref{eq:cr_subproblem}
and restated here for convenience:
\begin{equation}
    \underset{s \in \mathbb{R}^n}{\textrm{minimize }} 
    m_k(s) \defeq f(x_k) + s^Tg_k + \frac{1}{2} s^TB_ks + \frac{1}{3} \sigma_k \| s \|^3. \tag{\ref{eq:cr_subproblem}}
\end{equation}
We start by describing a Newton-based approach proven to be convergent
in~\citet{ARCP1}.  Though their approach targets dense matrices $B_k$ where
Cholesky factorizations are viable, we subsequently show in this section how to efficiently
extend this approach to large-scale limited memory QN matrices.

The Newton-based solver for Equation~\ref{eq:cr_subproblem} is based on the following 
theorem:
\begin{theorem}[\cite{ARCP1}]\label{thm:firstorderCR}
Let $B_k(\lambda) \defeq B_k+\lambda I,$ $\lambda_1$ denote the smallest eigenvalue of $B_k,$
and $u_1$ its corresponding eigenvector.
A step $s^\ast_k$ is a global minimizer of $m_k(s)$ if there exists 
a $\lambda^\ast \ge \max(0, -\lambda_1)$ such that: 
\begin{eqnarray}
B_k(\lambda^\ast)s_k^\ast &=&   -g_k,    \label{eq:optssol} \\
\|s^\ast_k\| &=& \dfrac{\lambda^\ast}{\sigma_k} , \label{eq:slambdasigma}
\end{eqnarray}
implying $B_k(\lambda^\ast)$ is positive semidefinite. Further, only if $B_k$ is indefinite, $u_1^T g =0,$ and $\|(B+\lambda_1 I)^\dagger g\| \le -\lambda_1/\sigma_k$, then $\lambda^* = -\lambda_1.$
\end{theorem}

We can then see that for the case where $\lambda^\ast > -\lambda_1$, $s^\ast_k$ is given by the solution to the following equation: 
\begin{equation}
    \phi_1(\lambda) \defeq \frac{1}{\vert \vert s(\lambda) \vert \vert} - \frac{\sigma}{\lambda} = 0\label{eq:optimality}.
\end{equation}
For simplicity we define $s(\lambda) \defeq -(B+\lambda I)^{-1} g$, where the pseudo-inverse is used for the 
case where $\lambda=-\lambda_1.$  Note the authors of \citet{ARCP1} show that when $B_k$ indefinite and $u_1^Tg = 0$, the solution $s^\ast_k$
is given by $s^*_k = s(-\lambda_1) + \alpha u_1$ where $\alpha$ is a solution to the equation
$-\lambda_1 = \sigma \vert \vert s(-\lambda_1) + \alpha u_1 \vert \vert.$  That is, whenever 
Equation~\ref{eq:optimality} fails to have a solution, $s^\ast_k$ is obtained by adding
a multiple of the direction of greatest negative curvature to the min-two norm solution to Equation~\ref{eq:optssol} 
so that Equation~\ref{eq:slambdasigma} is satisfied.
The authors of~\cite{ARCP1} thus apply Newton's method to $\phi_1(\lambda)$ resulting in Algorithm~\ref{algo:subproblem}.
This corresponds to Algorithm (6.1) of \citet{ARCP1}.
\begin{algorithm}
\begin{algorithmic}
    \caption{Newton's method to find $s^*$ and solve $\phi_1(\lambda) = 0$}\label{algo:subproblem}
    \If{$B$ indefinite, $u_1^T g = 0$}
        \If{$\|s(-\lambda_1)\| < \frac{\lambda}{\sigma}$}
            \State  Solve $-\lambda_1 = \sigma \vert \vert s(-\lambda_1) + \alpha u_1 \vert \vert$
            \State  $s^* \gets s(-\lambda_1) + \alpha u_1$
        \Else 
            \State $s^* \gets s(-\lambda_1)$
        \EndIf
    \Else
     \State Let $\lambda > \max(0, -\lambda_1)$.
        \While {$\phi_1(\lambda) \neq 0$}
                \State Solve 
                
                \begin{equation}\label{eq:s_lambda}
                    (B+\lambda I) s = -g.
                \end{equation}
                \State Let $B + \lambda I = LL^T$. Solve 
                \begin{equation}\label{eq:calc_w}
                    Lw = s.
                \end{equation}
                \State Compute the Newton correction 
                
                \begin{equation}\label{eq:dellamN}
                    \Delta \lambda^N \defeq \frac{\lambda \big( \vert \vert s \vert \vert - \frac{\lambda}{\sigma}  \big)}{\vert \vert s \vert \vert + \frac{\lambda}{\sigma} \big( \frac{\lambda \vert \vert w \vert \vert^2}{\vert \vert s \vert \vert^2} \big)}
                \end{equation}
                
                \State Let $\lambda \gets \lambda + \Delta \lambda^N$.
        \EndWhile
        
        \State $s^* \gets s(\lambda)$
    \EndIf
\end{algorithmic}
\end{algorithm}

At first glance, Algorithm~\ref{algo:subproblem} may not look feasible, as equation~\ref{eq:calc_w} requires the Cholesky matrix $L$, which is only cheaply obtained for small dense systems. 
Looking closer, we note that to execute Algorithm~\ref{algo:subproblem} one need not form $s$ and $w$, as only their corresponding norms are needed to compute $\Delta \lambda^N.$
Relevantly, it has been demonstrated that matrices in the quasi-Newton family have compact
matrix representations for the form
\begin{equation}
B=\gamma I + \Psi M^{-1} \Psi^T,  \label{eq:12}
\end{equation}
further detailed in \citet{SQN}. For example, for LSR1 $\Psi = Y - \gamma S$ and $M=(E - \gamma S^T S)$
where $E$ is a symmetric approximation of the matrix $S^T Y,$ whose lower triangular elements
equal those of $S^T Y$~\cite{erway2015efficiently}.  
They further show that for matrices of this class, an $\mathcal{O}(mn)$ calculation
may be used to implicitly form the spectral decomposition 
\[
B=U \Lambda U^T,
\] 
where $U$ is never formed but stored implicitly and $\Lambda$ satisfies
\begin{equation}\label{eq:qndiag}
\Lambda = \begin{pmatrix}
\gamma I & 0 \\
0 &   \hat \Lambda 
\end{pmatrix}
\end{equation}
where $\hat \Lambda \in \mathbb{R}^{k \times k}$ is the diagonal matrix defined in \citet{erway2015efficiently}. 
Thus we know that $B$ will have a cluster of eigenvalues equal to $\gamma$ of size $n-k.$ We can 
exploit this property to further reduce the computational complexity of Algorithm~\ref{algo:subproblem}.
Using this decomposition we see that $\Delta \lambda^N$ in Equation~\ref{eq:dellamN} can be computed  
as long as $\|s\|$ and $\|w\|$ are known.  Using the eigen decomposition of $B_k$ we get
\begin{align*}
\|s\|^2 &= g^T U (\Lambda +\lambda I)^{-2} U^T g \\
\|w\|^2 &= s^T L^{-T} L^{-1} s = s^T (B+\lambda I)^{-1} s = g^T U (\Lambda +\lambda I)^{-3} U^T g.  
\end{align*}
Note here that $\lambda$ denotes the parameter optimized in Algorithm~\ref{algo:subproblem} and
\emph{not} the diagonal values of $\Lambda.$ If we then define $U$ block-wise we can define the components $\hat g_1$ and $\hat g_2$
 as follows:
\begin{equation}
U \defeq \begin{pmatrix} U_1 & U_2 \end{pmatrix} \Rightarrow U^T g = \begin{pmatrix} U_1^T g \\ U_2^T g \end{pmatrix} 
= \begin{pmatrix} \hat g_1 \\ \hat g_2 \end{pmatrix}. \label{eq:Ugdef}
\end{equation}
Thus we can compute $\|s\|$ and $\|w\|$ in $\mathcal{O}(m)$ operations assuming $\hat g$ is stored, giving 
\begin{align}
\|s\|^2 &= \dfrac{\|\hat g_1\|^2}{(\lambda + \gamma)^2} + \sum_{i=1}^m \dfrac{\hat g_2(i)^2}{(\hat \lambda_{i} + \lambda)^2}\label{eq:snorm} \\
\|w\|^2 &= \dfrac{\|\hat g_1\|^2}{(\lambda + \gamma)^3} + \sum_{i=1}^m \dfrac{\hat g_2(i)^2}{(\hat \lambda_{i} + \lambda)^3}  \label{eq:wnorm}
\end{align}
Using this, the computation cost of Newton's method is reduced to an arguably to an inconsequential amount,
assuming that $\|\hat g_1\|$ and $\hat g_2$ required by Equations~\ref{eq:snorm} and~\ref{eq:wnorm}
can be efficiently computed.  We dub this optimization the ``norm-trick". We additionally
note that $Y$ and $S$ both change by only one column each iteration
of Algorithm~\ref{algo:full} (defined later).  Thus with negligible overhead we can
cheaply update the matrix $\Psi^T \Psi \in \mathbb{R}^{m \times m}$ each iteration by
retaining previously computed values that are not dependent on the new $(s,y)$ pair.  
Making the following two assumptions we
can next show that Algorithm~\ref{algo:subproblem} can be solved with
negligible overhead compared to classical LQN approaches.  

More specifically
we note that classic
LQN methods at each iteration must update $B_k$ and then solve a system of the form $s=-(B_k+\lambda I)^{-1}g$ for some $\lambda \ge 0.$
This creates an $\mathcal{O}(mn)$ computational lower bound that we aim to likewise achieve
when generating an optimal step for Equation~\ref{eq:cr_subproblem}.
In contrast to the approach described in~\citet{ARCsLSR1} that uses a QR factorization
of $\Psi$, we use an analogous approach to that described in~\citet{burdakov2017efficiently} for
trust-region methods to perform the majority of the required calculations
on matrices of size $m \times m$ in place of $n \times m.$  This saves 
significantly on both storage computational overhead.  
Note that unlike~\citet{burdakov2017efficiently} we do not explicitly compute $M^{-1}$
as we have found this matrix can become ill-conditioned periodically. We explain
how we are able to do this constructively in the following theorems.

\begin{assumption}\label{assume:PsiPsi}
The matrix $T=\Psi^T \Psi$ is stored and updated incrementally.  That is,
if $\Psi$ has one column replaced, then only one row and column of $T$ is updated.
\end{assumption}
\begin{assumption}\label{assum:Psig}
The vector $\bar u=\Psi^T g$ is computed once each iteration of Algorithm~\ref{algo:full}
and stored.
\end{assumption}
The following theorem shows that given $T$ and $\bar u$ defined in Assumptions~\ref{assume:PsiPsi}
and~\ref{assum:Psig}, that the optimal $\lambda^\ast$ can be cheaply obtained by Algorithm~\ref{algo:subproblem}
using $\mathcal{O}(m^3)$ operations.
\begin{theorem} \label{thm:solving}
Suppose that $B=\gamma I + \Psi M^{-1} \Psi^T$ as defined in Equation~\ref{eq:12}
and that $(V,\Lambda)$ 
solves the generalized eigenvalue problem $Mv = \lambda T v.$
Then $U_2$ as defined in \ref{eq:Ugdef} is given by
$U_2 = \Psi V.$  Further the corresponding eigenvalues $\hat \Lambda$ from Equation~\ref{eq:qndiag}
are given by $\hat \Lambda= (\gamma I + \Lambda^{-1}).$ We can then recover
$\hat g_2 = V^T \bar u,$
and $\|\hat g_2\| = g^T g - \hat g_2^T \hat g_2.$
\end{theorem}
\begin{proof}
Rather than inverting $M$
we can simply solve the generalized eigenvalue problem 
$[V,\Lambda]={\rm eig}(M, \Psi^T \Psi),$
where
\begin{eqnarray*}
MV &=& \Psi^T \Psi V \Lambda \\
V^T M V &=& \Lambda \\
V^T \Psi^T \Psi V &=& I,
\end{eqnarray*} 
where $\Lambda$ is the diagonal matrix of
generalized eigenvalues for the system $Mv=\lambda \Psi^T \Psi v.$
Then we have $U_2 = \Psi V$ implying
\[
B U_2 = \gamma  U_2 + \Psi M^{-1} \Psi^T (\Psi V \Lambda \Lambda^{-1})=
\gamma U_2 + \Psi M^{-1} M V \Lambda^{-1} =  U_2 (\gamma I + \Lambda^{-1}).
\]
Thus we can can set $\hat \Lambda$ from equation~\ref{eq:qndiag}
as $\hat \Lambda= (\gamma I + \Lambda^{-1}).$ 
Further we can recover
$\hat g_2 = U_2^T g =V^T \Psi^T g = V^T \bar u$
and $\|\hat g_2\| = g^T g - \hat g_2^T \hat g_2.$
\end{proof} 

Using the previous theorems and Equations~\ref{eq:snorm} and~\ref{eq:wnorm}
we can thus obtain $\lambda^\ast=\sigma \|s^\ast\|$ from Algorithm~\ref{algo:subproblem}
in $\mathcal{O}(m^3)$ additional flops once $T$ and $\bar u$ are formed.
We now show how to efficiently recover the optimal $s^\ast$ 
from Equation~\ref{eq:cr_subproblem} using $\mathcal{O}(mn)$ flops.
\begin{theorem}
Using the same assumptions and definitions as in Theorem~\ref{thm:solving},
given any $\lambda > {\rm max}(0,-\lambda_1),$ the solution $s=-(B+\lambda I)^{-1} g$ is given
by
\[
-\dfrac{1}{\lambda +\gamma}g  -\Psi V r,
\]
where $r$ can be formed with $\mathcal{O}(m^2)$ computations. 
\end{theorem}
\begin{proof}
 Note we can further
save on computation by storing $\Psi^T g$ for later calculations when
we recover the final search direction.  Note that the very end we must form the
search direction by solving the system $(B + \lambda I)s =-g $
for the optimal value of $\lambda.$  This implies
\begin{eqnarray*}
s&=&-U (\Lambda+\lambda I)^{-1}U^T g \\
&=& -\dfrac{1}{\lambda +\gamma} U_1 U_1^T g - U_2 (\hat \Lambda + \lambda I)^{-1} U_2^T g \\
&=&-\dfrac{1}{\lambda +\gamma} \left( U_1 U_1^T g + (U_2 U_2^T g - U_2 U_2^Tg)\right) -U_2 (\hat \Lambda + \lambda I )^{-1} U_2^T g \\
&=& -\dfrac{1}{\lambda +\gamma}g  -U_2 
[(\hat \Lambda + \lambda I)^{-1} -\dfrac{1}{\lambda +\gamma} I]
U_2^T g \\
&=& -\dfrac{1}{\lambda +\gamma}g  -\Psi V \bar E V^T \Psi^T g,
\end{eqnarray*}
where $E = (\hat \Lambda + \lambda I)^{-1} - (\lambda + \gamma)^{-1} I) .$ 
\end{proof}
\begin{theorem}
Let $(\lambda_1, u_1)$ denote the eigenpair corresponding to the most negative
eigenvalue of the matrix $B.$ Then, if $\gamma < \min(\text{{\normalfont diag}}(\hat \Lambda)),$ $u_1$
can be formed as $u_1 = \hat r/\|\hat r\|$ where $\hat r = (I - U_2 U_2^T) r$ for any vector $r$ in $\mathbb{R}^n$ such
that $\|\hat r\| > 0.$  Otherwise $u_1=\Psi v_k$ where $v_k$ is a column of $V$
that corresponds to the smallest eigenvalue of $\hat \Lambda.$
\end{theorem}
\begin{proof}

Note that $(I-U_2 U_2^T)$ is the projection matrix onto the subspace defined by the ${\rm span}(U_1)$
implying $U_2 \hat r = 0.$
\[
B \hat r = \gamma U_1 U_1^T\hat r = \gamma (I-U_2 U_2^T) \hat r = \gamma \hat r, 
\]
since $\hat r$ has already been projected.  Thus $\hat r$ is an eigenvector of $\gamma.$ 
If $\gamma$ is not the smallest eigenvalue of $B,$ then by design $u_1$ can be obtained from $U_2e_1$ assuming
the eigenvalues of $\hat \Lambda$ are sorted smallest to largest.
\end{proof}

Note that in practice we only need to form $u_1$ explicitly if $\lambda^\ast = \max(0, -\lambda_1).$
This can only be true if $-\lambda_1 < 0$ and $\|s(-\lambda_1)\| < -\lambda_1/\sigma_k.$  This condition
is again cheaply checked beforehand using Equation~\ref{eq:snorm}. Hence the primary computational cost of an iteration of Algorithm~\ref{algo:subproblem} is
the work needed to update a single row and column of $\Psi^T \Psi$ and $M$ as
well as form $\Psi^T g$ and $\Psi r$ once.  This is the same work required by classical quasi-Newton
methods as well.

Using the techniques developed and proven in this section, we can now form a very efficient solver for Equation~\ref{eq:cr_subproblem} (using a modified version of Algorithm~\ref{algo:subproblem}); using the norm-trick, we can avoid explicitly forming $s$ and $w$, reducing complexity of a Newton iterate from $\mathcal{O}(mn)$ to $\mathcal{O}(m)$.
\subsection{Solving the nonlinear optimization problem}
In this section we focus on solving the problem
\begin{equation}\label{eqn-min}
\underset{x\in {\mathbb{R}}^n}{\textstyle{\min}} f(x) = \textstyle{\sum}_{i=1}^N f_i(x),
\end{equation}
where $f_i(x)$ is defined as the loss for the $i$-th datapoint, using the subproblem solver defined in Section~\ref{sec:subsolver}. 
We follow the framework described in \citet{ARCP1}. 
A benefit of the algorithm defined in Algorithm~\ref{algo:simple_ARC} is that first-order convergence is proven
if $B_k$ remains bounded and $f(x) \in \mathcal{C}^1(\mathbb{R}^n)$.
Thus the condition that $B_k=\nabla^2 f(x)$ is greatly relaxed from its
predecessors such as~\citet{nesterov2006cubic}.
\begin{algorithm}
    \caption{Adaptive Regularization using Cubics (ARC). \textcolor{blue}{Blue text} indicates our modification to default to an SGD-like step on failure.}\label{algo:simple_ARC}
    

    Given $x_0, \sigma_0 > 0, \gamma_2, \gamma_1, \eta_2 > \eta_1 > 0, \alpha > 0$, for $k = 0, 1, \dots, $ until convergence,
    
    \begin{enumerate}
        \item 
            Compute update $s^*_k$ such that:
            \begin{equation}
                m_k(s^*_k) \le m_k(s^c_k)\label{eq:cauchy_update}
            \end{equation}
            where the Cauchy point $s^c_k = -\upsilon^c_k g_k$ and $\upsilon^c_k = \underset{\upsilon \in \mathbb{R_+}}{\argmin\,\,} m_k( -\upsilon g_k).$
            
        \item Compute ratio between the estimated reduction and actual reduction
        \begin{equation}
            \rho_k \gets \frac{f(x_k) - f(x_k + s^*_k)}{f(x_k) - m_k(s^*_k)}\label{eq:rho_calc}
        \end{equation}

        \item Update
        \begin{equation}
            x_{k + 1} \gets \begin{cases} 
            x_k + s^*_k & \text{if } \rho_k \ge \eta_1 \\
            x_k  \textcolor{blue}{- \alpha g_k} & \text{otherwise}
            \end{cases}\label{eq:update_acceptance}
        \end{equation}
        
        \item Set
        \begin{equation}
            \sigma_{k + 1} \text{ in } \begin{cases} 
            (0, \sigma_k] & \text{if } \rho_k > \eta_2 \\
            [\sigma_k, \gamma_1 \sigma_k] & \text{if } \eta_2 \ge \rho_k \ge \eta_1 \\
            [\gamma_1 \sigma_k, \gamma_2 \sigma_k] & \text{otherwise}
            \end{cases}\label{eq:rho_update}
        \end{equation}
    \end{enumerate}
\end{algorithm} 

In Algorithm~\ref{algo:simple_ARC}, we first solve the CR subproblem (Equations~\ref{eq:cr_subproblem},\ref{eq:cauchy_update}; Algorithm~\ref{algo:subproblem}) to find our step, $s^*_k$. We then determine if the step is accepted by examining if the ratio between the decrease in the objective ($f(x_k) - f(x_k + s^*_k)$) and the predicted decrease in objective ($f(x_k) - m_k(s^*_k)$) is large enough (Equations~\ref{eq:rho_calc}-\ref{eq:update_acceptance}), determined by the hyperparameter $\eta_1$. Then, depending on the change of performance and choice of $\eta_2$, we adjust our regularization parameter $\sigma_k$: the `better' the step, the more we decrease $\sigma_{k+1}$, the worse, the more we increase it (Equation~\ref{eq:rho_update}). The amount of increase and decrease is governed by two hyperparameters, $\gamma_1$ and $\gamma_2$.

We note one important modification to the ARC framework: if we find that $\rho < \eta_1$, we take an SGD step instead of just setting $x_k \gets x_{k - 1}$ (Equation~\ref{eq:update_acceptance}). While, empirically, rejected steps are not common, we find that reverting to SGD in case of failure can save time in cases where $B_k$ is ill-conditioned. One may note that we have no guarantees that $f(x_k) - f(x_k - \alpha g_k) > 0$, which may seem to contradict the ARC pattern detailed in \citet{ARCP1} which only accepts steps which improve loss. However \citet{Chen2018} proves that in a trust-region framework, if you accept all steps, $\rho_k$ only need be positive half of the time for almost sure convergence (\citet{LinesearchConvergance} proves a similar result for first-order methods); it has also been shown that noisy SGD steps improve performance of final solution quality \cite{RethinkingSGD, SGDRegularization}. Implementation details regarding this choice can be found in Section~\ref{section:implementation_details}.

Putting this together, we can form the full ARCLQN algorithm, detailed in Algorithm~\ref{algo:full}.

%

\algrenewcommand\algorithmicensure{\textbf{Output:}}
\begin{algorithm}
\caption{ARCLQN, our proposed algorithm for solving \ref{algo:simple_ARC} under memory constraints.}\label{algo:full}
\begin{algorithmic}[1]
\Require{Given $x_0$ : initial parameter vector}
\Require{$0 < \eta_1 \le \eta_2$ : hyperparameters to measure the level of success of a step}
\Require{$\mathcal{D}$, $k$ : dataset and minibatch-size, respectively.}
\Require{$\sigma_0$ : starting regularization parameter}
\Require{$\epsilon, \delta$ : tolerance parameters }
\Require{$f(x, b)$ : objective function with inputs parameters $x$ and minibatch $b$}
\Require{$\alpha_1, \alpha_2$ : learning rates}
\State Initialize $B_0 = I$.
\For{$k = 1, 2, \dots$}
    \State Let $b_k$ be a minibatch sampled randomly from $\mathcal{D}$ of size $k$
     \State $g_k \gets \nabla_x f(x_{k-1}, b_k)$ 
    \State Calculate $\lambda_1$ of $B_{k-1}$ 
    \State Let $\lambda \gets \max(-\lambda_1, 0) + \epsilon$
    \State Compute $s^*_k$ (using Algorithm~\ref{algo:subproblem})
    \State Calculate $\rho$ (as in Equation~\ref{eq:rho_calc})
    \If{$\rho \ge \eta_1$}\label{algoline:acceptance}
        \State $x_k \gets x_{k - 1} + \alpha_1 s^*_k$
        \State $y \gets \nabla_x f(x_k, b_k) - g_k$
        \State Update $B_k$ using $B_{k - 1}, \alpha_1 s, y$ if update and resulting $B_k$ are well-defined\label{algoline:bkupdate1}
        \If{$\rho \ge \eta_2$}
            \State $\sigma_k \gets \max(\frac{\sigma_{k - 1}}{2}, \delta)$
        \EndIf 
    \Else
        \State $\sigma_k \gets 2\cdot\sigma_{k-1}$
        \State $x_k \gets x_{k - 1} - \alpha_2 g_k$
        \State $y \gets \nabla_x f(x_k, b_k) - g_k$
        \State Update $B_k$ using $B_{k - 1}, -\alpha_2 g_k, y$ if update and resulting $B_k$ are well-defined\label{algoline:bkupdate2}
    \EndIf 
\EndFor
\end{algorithmic}
\end{algorithm}

\section{Numerical Results}
\subsection{Comparison to SR1}

We start by benchmarking the optimized CR subproblem solver alone, without integration into the larger ARCLQN optimizer. These results are summarized in Table~\ref{table:1}. All timing information is reported as the average across 10 runs. We can see that the dense SR1 solver fails to scale to more than 10,000 variables. We also see that the traditional LSR1 solver becomes 
computationally prohibitive for higher dimensions.
For example, when $n=10^{8},$ the positive-definite test case 
takes 274 seconds to converge, whereas following steps outline in Section~\ref{sec:subsolver} it is reduced to 2.33 seconds, a speedup of over 100x. Considering the CR subproblem represents the bulk of the computation of any given optimization step, this performance 
improvement greatly increases the scalability of the algorithm.
In the next section, we use the enhancements highlighted
here to provide preliminary results using Algorithm~\ref{algo:full}.

\begin{table}[!htbp]
\centering
\small \setlength{\tabcolsep}{5pt}
\begin{tabular}{l *{9}{l}}
\toprule
 & & \multicolumn{6}{c}{Time (in seconds) to solve CR subproblem of given dimension} \\
 \cmidrule(r){3-9}
 Method & Case & $1\text{e}2$ & 1e3 & 1e4 & 1e5 & 1e6 & 1e7 & 1e8\\
 \midrule
 SR1 & Hard & $3.71\text{e-}3$ & $2.79\text{e-}1$ &  $1.03\text{e}2$ &  - &  - & - & - \\
 LSR1 & Hard & $4.47\text{e-}4$ &  $8.05\text{e-}4$ &  $3.16\text{e-}3$ & $5.72\text{e-}3$  & $4.05\text{e-}2$ & $8.75\text{e-}1$ & $7.41\text{e}0$\\
 ARCLQN & Hard & $4.31\text{e-}4$ & $7.34\text{e-}4$ & $1.68\text{e-}3$ & $4.47\text{e-}3$ & $2.66\text{e-}2$ & $6.44\text{e-}1$ & $5.64\text{e}0$\\
    \addlinespace 
 SR1 & Indefinite & $1.93\text{e-}3$ &  $8.53\text{e-}2$ &  $1.17\text{e}1$ &  - & - & - & -\\
 LSR1 & Indefinite & $1.49\text{e-}3$ & $3.54\text{e-}3$ &  $1.90\text{e-}2$ &  $4.40\text{e-}2$ &  $7.78\text{e-}1$ & $8.55\text{e}0$ & $8.14\text{e}1$ \\
 ARCLQN & Indefinite & $8.02\text{e-}4$ &  $1.69\text{e-}3$ &  $1.99\text{e-}3$ &  $2.79\text{e-}3$ & $1.90\text{e-}2$ & $2.59\text{e-}1$ & $2.39\text{e}0$\\
    \addlinespace 
 SR1 & Positive Definite & $1.39\text{e-}3$ &  $9.65\text{e-}1$ &  $9.23\text{e}1$ &  - & - & - & -\\
 LSR1 & Positive Definite & $7.85\text{e-}3$ & $1.64\text{e-}2$ & $9.09\text{e-}2$ & $1.48\text{e-}1$ & $2.12\text{e}0$ & $3.06\text{e}1$ & $2.74\text{e}2$\\
 ARCLQN & Positive Definite & $3.64\text{e-}3$ & $6.03\text{e-}3$ & $6.45\text{e-}3$ & $7.90\text{e-}3$ & $2.68\text{e-}2$ & $2.70\text{e-}1$ & $2.33\text{e}0$\\
\bottomrule
\end{tabular} 
\newline
\caption{Timing information for solving the CR subproblem, Equation~\ref{eq:optimality}. A hyphen indicates that the test did not terminate within 300 seconds. SR1 corresponds to an optimized SR1 implementation. LSR1 corresponds to an optimized LSR1 implementation, without the norm-trick. ARCLQN corresponds to an optimized LSR1 implementation that uses the norm-trick. For limited memory experiments, $m = 3$ was used. Cases are detailed in section~\ref{section:ARC}. All other columns correspond to the problem dimension, and entries correspond to time (in seconds) required to find the global minimizer $s^*$ using CPU.}\label{table:1}
\end{table} 
\raggedbottom
\subsection{Autoencoding}

We also experiment using our ARCLQN solver as the optimizer of an autoencoder, as detailed in \citet{goodfellow2016deep}. As our dataset, we use CIFAR-10 \cite{CIFAR10}. We compare against a number of recent optimizers \cite{Adam, SGDoverview, APOLLO, AdaHessian}. For this experiment, as our Hessian approximation, we use an LSR1 matrix \cite{LSR1}. We performed a hyperparameter search over learning rates in $\alpha \in \{.001, .01, .1, .7, 1\}$ for all optimizers. For applicable optimizers, we also searched for the optimal $\epsilon \in \{1\text{e-}4, 1\text{e-}8, 1\text{e-}16\}$. For each optimizer, we chose the hyperparameter settings that lead to the lowest test loss after 10 epochs. For limited memory methods, we fixed the history size at $m = 5$. We use a convolutional neural network \cite{goodfellow2016deep}, with 3 convolutional layers, then 3 transposed convolutional layers. For all layers, padding and stride are set to 1 and 2 respectively. In-between all layers (except the middle one), SeLU \cite{SELU} activation are used; the final layer uses a sigmoid activation. Layers have $3, 12, 24, 48, 24, 12$ input channels respectively; minibatch size is fixed as 128.  As a loss function, we use binary cross entropy; for training, we take the mean across the batch, for evaluation, we take the sum across the entire dataset. For a more detailed look at the hyperparameters, please read Section~\ref{section:hyperparameters}. Results are summarized in Figure~\ref{fig:cifar_loss} and Table~\ref{table:cifar_speed}. All numbers reported are averaged across 10 runs. It is worth noting that while LBFGS converges more rapidly, it is over twice as slow as our approach, and suffers from numerical stability issues: of the 10 runs performed, 2 failed due to NaN loss. 

\begin{minipage}{\textwidth}
\begin{minipage}[!t]{0.55\textwidth}
\centering
\includegraphics[scale=0.55]{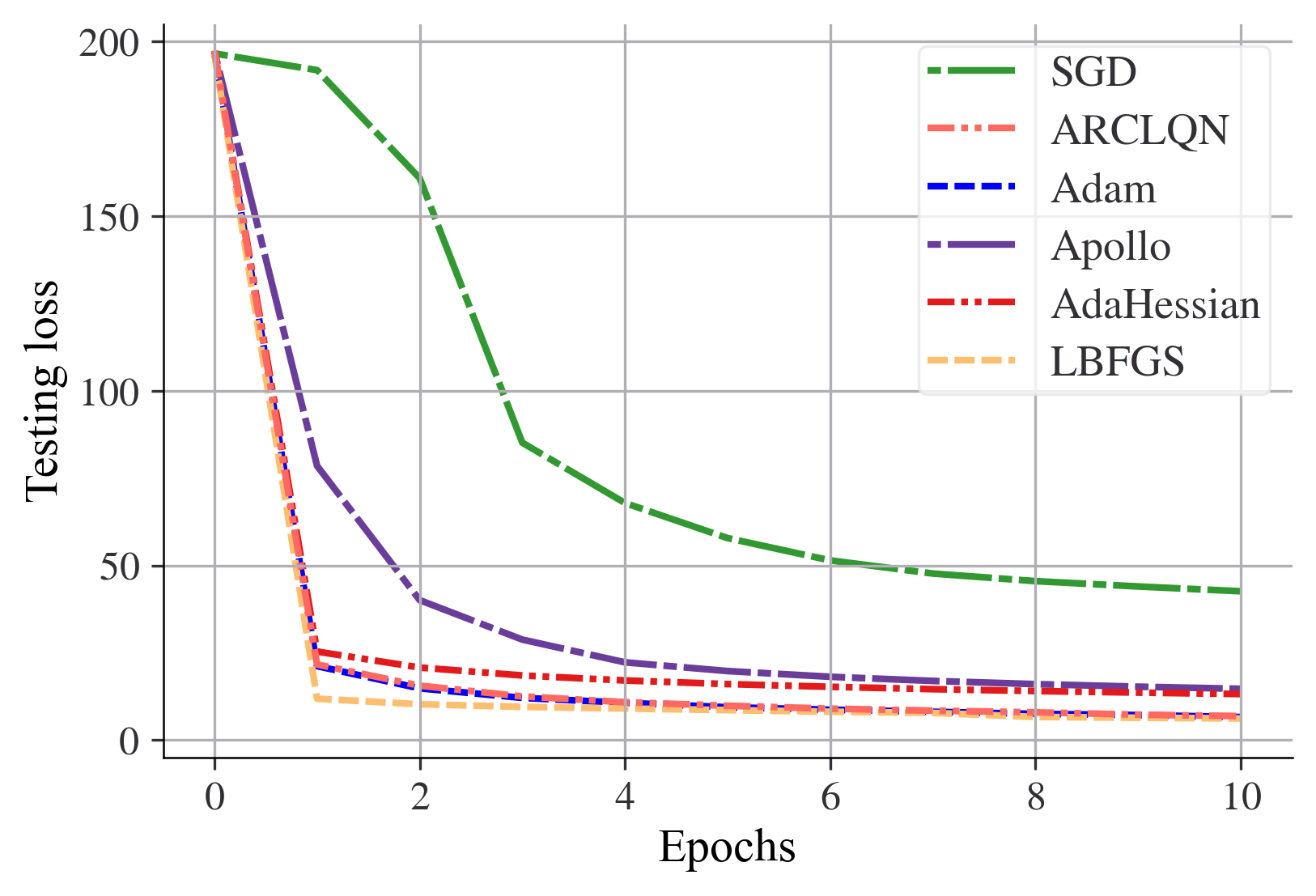}
\captionof{figure}{Testing loss of the trained CIFAR-10 autoencoder, evaluated at the end of each epoch.}\label{fig:cifar_loss}
\end{minipage}
\hfill
\begin{minipage}[!t]{0.44\textwidth}
\centering
\begin{tabular}{cc}\hline
    Optimizer & Cost ($\times$SGD)\\ \hline
    SGD        & 1.00 \\
    Adam       & 1.00 \\
    Apollo     & 1.00 \\
    AdaHessian & 1.01 \\
    ARCLQN    & 2.56 \\
    LBFGS      & 5.45 \\ \hline
  \end{tabular}
  \captionof{table}{The amount of time taken to train for 10 epochs (compared to SGD) for different optimizers. Timing information included all parts of training, but excluded calculation of test set loss.}\label{table:cifar_speed}
\end{minipage}
\end{minipage}

\section{Conclusion} 

We have introduced a new family of optimizers, collectively referred to as ARCLQN, which utilize a novel fast large-scale solver for the CR subproblem. We demonstrate very large speedups over a baseline implementation, and benchmarked ARCLQN, finding it competitive with Adam and LBFGS, and beating other strong optimizers. To the best of our knowledge, ARCLQN is the first extension of the ARC framework to the limited memory case without major modification of the core framework. We additionally expand upon ARC, explicitly incorporating first-order updates into our methodology. 

\FloatBarrier
\bibliography{refs}

\appendix

\section{Implementation Details}\label{section:implementation_details}

Second order methods can, at times, be unstable. To achieve good performance and stable training, it is important to use heuristics to prevent or alleviate this instability. For the sake of complete transparency, we share all used heuristics and modifications not explicitly detailed elsewhere in this section. It is worth explicitly noting that none of the parameters below have been tuned for performance, and instead have been chosen either completely arbitrarily, or via test-runs on toy problems. There may be significant room for improvement with actual tuning of these parameters, and we leave that to future work.

If the step we take multiple times times in a row are very similar to each-other, or if the steps are very small, we can run into issues with $S_k$ being singular or $B_k$ being ill conditioned. We use two heuristics to detect and fix this. First, before updating $B_k$ on lines~\ref{algoline:bkupdate1}~and~\ref{algoline:bkupdate2} of Algorithm~\ref{algo:full}, we set $y \gets \frac{y}{\max(\|s\|, \kappa)}$ and $s \gets \frac{s}{\max(\|s\|, \kappa)}$. This prevents $B_k$ from becoming ill-conditioned if $\| s \|$ is very small. Additionally, we also reset $B_k$ if the minimum eigenvalue of $S_k^TS_k$ is less than $\kappa$. When we `reset' $B_k$, we drop the first and last column of $S_k$ and $Y_k$ instead of setting $B_k \gets I$; this helps us prevent resetting from destroying too much curvature information. We set $\kappa = 1\text{e-}7$.

Another important detail is that we do not check if $\phi_1(\lambda) = 0$ is true when using Newton's method in Algorithm~\ref{algo:subproblem}. Instead, we repeat the while loop until $\|s\|-\frac{\lambda}{\sigma} < \nu$. For CIFAR-10 autoencoding, we set $\nu = 1\text{e-}5$. For comparison to SR1, we set $\nu = 1\text{e-}7$. We empirically find that, for larger scale problems, $\nu$ can be set higher, as $\phi_1(\lambda)$ does not change much at the final iterations of Algorithm~\ref{algo:subproblem}. We also modify the update criteria, detailed in Equation~\ref{eq:update_acceptance} and line~\ref{algoline:acceptance} in Algorithm~\ref{algo:full}: we check not only that $\rho \ge \eta_1$ but also that $f(x_k, b_k) - f(x_k + s^*_k, b_k)$ is greater than $\mu$. For CIFAR-10 experiments, we set $\mu = 1\text{e-}3$. This prevents issues of $\|s\| \to 0$ as $\sigma_k \to \infty$. Since we update $B_k$ regardless of which type of step we take, we empirically found that updating $B_k$ with $s, y$ pairs from SGD-like steps when $(x_k, b_k) - f(x_k + s^*_k, b_k)$ is very small leads to higher $\rho$ subsequently. For CIFAR-10 experiments, we do not take an actual SGD step, but instead, use an Adam step \cite{Adam}. Finally, we also bound $\sigma$ above by 8096, as rarely $\rho < \eta_1$ can occur multiple times in a row, which can lead to many very small steps being taken with little affect on performance.

\section{Hyperparameters}\label{section:hyperparameters}

\begin{center}
\begin{tabular}{cccc}\hline
    Optimizer & $\alpha$ & $\beta$ & $\epsilon$\\ \hline
    SGD        & $0.01$ &  $0.9$ & - \\
    Adam       & $0.001$ & $(0.9, 0.999)$ & $1\text{e-}8$ \\
    Apollo     & $0.01$ & $0.9$ & $1\text{e-}8$ \\
    AdaHessian & $0.1$ & $(0.9, 0.999)$ & $1\text{e-}4$ \\
    ARCLQN    & $(1, 0.001)$ & $(0.9, 0.999)$ & $1\text{e-}4$\\
    LBFGS      & $1$ & - & - \\ \hline
  \end{tabular}\captionof{table}{Hyperparameter settings for the optimizers used in the CIFAR-10 autoencoding experiments. A dash indicates that the optimizer does not have a given hyperparameter.}\label{table:hyperparams}
\end{center}

This section will detail optimizer settings not otherwise explicitly mentioned in the paper. The hyperparameters used to generate Figure~\ref{fig:cifar_loss} can be found in Table~\ref{table:hyperparams}. $\alpha_2 = 0.001$ for ARCLQN was set arbitrarily. For ARCLQN, we use $\eta_1 = 0.1, \eta_2 = 0.7$. These numbers were also set arbitrarily. For all optimizers with momentum, $\beta$ values were left at their defaults and not tuned. Since the Apollo paper emphasizes that warmup is extremely important for their method, we use linear warmup over 500 steps \cite{APOLLO}.

\end{document}